\documentclass[12pt]{article}
\usepackage{amssymb,amsmath,amsfonts,amsthm,amsbsy}

\ifx\pdfoutput\undefined
\usepackage{graphicx}
\else
\usepackage[pdftex]{graphicx}
\fi

\begin{document}

\newtheorem{theorem}{Theorem}
\newtheorem{lemma}[theorem]{Lemma}
\newtheorem{corollary}[theorem]{Corollary}
\newtheorem{proposition}[theorem]{Proposition}
\newtheorem{conjecture}[theorem]{Conjecture}

\theoremstyle{remark}
\newtheorem*{remark}{Remark}

\def\C{{\mathbb C}}
\def\F{{\mathbb F}}
\def\N{{\mathbb N}}
\def\Z{{\mathbb Z}}
\def\Q{{\mathbb Q}}
\def\G{{\mathcal G}}

\def\({\left(}
\def\){\right)}
\def\eps{\varepsilon}
\def\ind{\operatorname{ind}}

\def\spmod#1{\,(\text{\rm mod\,}#1)}
\def\sbmod#1{\,\text{\rm mod\,}#1}


\title{On the size of quotient of two subsets of positive integers.}

\author{\sc Yurii Shteinikov  
}

\date{}
\maketitle

\maketitle

\begin{abstract} We obtain non-trivial  lower bound for the set $A/A$, where $A$ is a subset of the interval [1, $Q$].
  
\end{abstract}

\begin{flushright}

{\em In a memory of }\\

{\em Anatoly Alekseevich Karatsuba}\\

{\em and Andrey Borisovich Shidlovsky}

\end{flushright}


\section{Introduction}
Let $A,B$ be  subsets of integers of the interval [1,$Q$], $|A|$ will denote the cardinality of finite set $A$. The sets $AB$ and $A/B$ are called the product and quotient of two sets  $A,B$ and are defined as 
$$AB= \{ab: a \in A, b \in B \}, A/B= \Bigl\{\frac{a}{b}: a \in A,  b \in B, b \neq 0 \Bigr\}.$$

Recall that the multiplicative energy  $E(A,B)$ of two sets $A,B$  is $$E(A,B)=|\{a_{1}b_{1}=a_{2}b_{2}: a_{1},a_{2} \in A; b_{1},b_{2} \in B \}|.$$ When the sets are equal, $A=B$ we will simply write $E(A)$ instead of $E(A,A)$. 

We note that using good estimates of $E(A,B)$ one can deduce non-trivial lower estimates of the size of $AB$ but not vice versa -- the following well-known inequality  which is due to Shnirel'man~\cite{lg}, which can be also found in~\cite{TV}.

\begin{equation} \label{shn}
|AB|, |A/B| \geq \frac{|A|^{2}|B|^{2}}{E(A,B)}.
\end{equation}

Throughout the paper   $\tau(n)$ (usual notation) is the number of  divisors of $n$. Recall the well-known estimate of $\tau(n)$  which can be found in the book  \cite{prach}, Theorem 5.2, Kapitel 1.

\begin{equation}
\tau(n) \leq 2^{\frac{(1+o(1))\log n}{\log \log n}}, n \rightarrow \infty.
\end{equation}

Using the above estimate  it is easy to prove the following result. 

\textit{For any finite set $A \subset \mathbb{N}$ such that  $ a \in A \Rightarrow a \leq Q,$ we have the following estimates}

\begin{equation}
|AA| \geq |A|^{2}\exp \Bigl\{ (-2 \log 2 +o(1)) \frac{  \log Q}{\log\log Q} \Bigr\}, Q\rightarrow \infty;
\end{equation}

\begin{equation}
E(A)\leq |A|^{2} \exp \Bigl\{ (2 \log 2 +o(1)) \frac{  \log Q}{\log\log Q} \Bigr\} Q\rightarrow \infty.
\end{equation}

The constant $2\log 2$ in (2),(3) can not be improved, see it in the paper  \cite{sht}  for example. 

Using (5) one can obtain that 
\begin{equation}
|A/A| \geq |A|^{2} \exp \Bigl\{ (-2 \log 2 +o(1)) \frac{  \log Q}{\log\log Q} \Bigr\} Q\rightarrow \infty.
\end{equation}

This bound  cannot be improved very much in general, except for the  constant~$-2 \log 2$, see it for example in \cite{crr}.

But there is  question that was posed in the paper \cite{II}  relating to this, -- we repeat the formulation  of it bellow.

\textbf{Question.} Is it possible to improve the coefficient $-2 \log 2$ in (4)?

The purpose of this note is to give positive answer to this question. So we formulate the main result of this  paper.

\begin{theorem} \label{t-1}

There is an absolute constant $\gamma>0$, such that  if $A,B \subseteq [1,Q]$ then we have the following estimate
$$|A/B|\geq |A||B|  \exp\Bigl\{(- 2 \log2 + \gamma + o(1)) \frac{\log Q}{{\log{\log Q}}}\Bigl\} ,  Q\rightarrow \infty.$$
One can take $\gamma=0.098$.
 
\end{theorem}

The notation $A\lesssim B$ in this paper denotes that $$A<B \exp\Bigl\{( o(1)) \frac{\log Q}{{\log{\log Q}}}\Bigl\} ,  Q\rightarrow \infty.$$

The paper is organized as follows: in the next section we formulate some preliminary statements. In the third section we give the proof of Theorem \ref{t-1}. The last section contains some final comments about this result.
\section{Preparations and preliminary results}

We need some definitions and preliminary lemmas.
We begin with the smooth numbers. For positive integer $n$ let $P^{+}(n)$ denotes the maximal prime divisor of $n$, and $P^{+}(1)=1$. For $x\geq y \geq 2$ let 

$$\psi(x,y)=|\{n \leq x: P^{+}(n)\leq y\}|.$$

We need some one upper estimate for $\psi(x,y)$, which can be found in \cite{ten}, Theorem 1.4, which is presented bellow.

\begin{lemma} \label{l-2}

Uniformly  for $x\geq y \geq 2$,  we have 

$$\log \psi(x,y) = Z\Bigl\{ 1 + O(\frac{1}{\log y} + \frac{1}{\log \log x}) \Bigr\}, $$
 where $$Z=Z(x,y)= \frac{\log x}{\log y} \log (1 + \frac{y}{\log x}) + \frac{y}{\log y}\log (1+ \frac{\log x}{y}).$$
\end{lemma}

Our second lemma gives some upper bound for the number of divisors of positive integer with small redical. Probably it was  known before and we do not pretend on this fact.

\begin{lemma} \label{l-3}

There exists a function $C(\varepsilon)>0,$ with  $C(\varepsilon) \longrightarrow 0, $ if $\varepsilon \longrightarrow 0$ and with the following property.  If $n\leq Q, rad (n) \leq Q^{\varepsilon}$ , then  $$\tau(n) \lesssim  \exp\Bigl\{(C(\varepsilon)) \frac{\log Q}{{\log{\log Q}}}\Bigl\} ,  Q\rightarrow \infty.$$  
\end{lemma}
\begin{proof}
Let $n=p_{1}^{\alpha_{1}}\ldots p_{s}^{\alpha_{s}}$  is the prime decomposition of $n$ and $p_{1}<p_{2}< \ldots <p_{s}.$  Consider the map on the set of divisors of $n$ $$\pi: p_{1}^{t_{1}}\ldots p_{s}^{l_{s}} \rightarrow  p_{(1)}^{t_{1}}\ldots p_{(s)}^{l_{s}},$$ 
where $p_{(i)}$ - is the $i$ - ordered prime number,--  $$p_{(1)}=2, p_{(2)}=3, p_{(3)}=5,...$$ By Prime Number Theorem if $rad(n) \leq Q^{\varepsilon} $ then $$p_{(s)} < (\varepsilon+o(1)) \log Q, Q\rightarrow \infty.$$ If $d|n $ then $\pi (d)\leq Q$. So the number of such divisors $d$ does not exceed $$\psi(Q, (\varepsilon+o(1) \log Q ).$$ Using Lemma \ref{l-2} with some easy computations we get the desired property for the function $C(\varepsilon)$.  This completes the proof of Lemma \ref{l-3}.
\end{proof} 


Let $\tau(n,z)$ denotes the number of divisors of $n$ which are less or equal to $z$. In other words 
$$\tau(n,z)=|\{ d: d|n, d \leq z \}|.$$

 The next proposition we present in the following lemma. 
\begin{lemma} \label{l-tau_mu}
Let $n\leq Q, \mu(n) \neq 0 $, $z\leq Q^{\delta}$ and $\delta \in (0,1/2]$. Then we have the following estimate $$\tau(n,z) \leq \exp\Bigl\{(\delta \log (\frac{1}{\delta}) + (1-\delta) \log (\frac{1}{1-\delta})+o(1)) \frac{\log Q}{{\log{\log Q}}}\Bigl\} ,  Q\rightarrow \infty. $$
\end{lemma}

\begin{proof} Consider any divisor $d$ of $n$ and its prime decomposition: $d=p_{1}\ldots p_{s}$. It is easy to see, that $s\leq \frac{(\delta + o(1)) \log Q}{{\log{\log Q}}}, Q~\rightarrow~\infty.$ The number $n$ has at most $\frac{ (1+o(1)) \log Q}{{\log{\log Q}}}, Q\rightarrow \infty$ different prime divisors. Doing some  computations together with the asymptotic expression for binomial coefficient -- $${m\choose \delta m}\sim \exp\Bigl\{(\delta \log (\frac{1}{\delta}) + (1-\delta) \log (\frac{1}{1-\delta})+o(1)) m \Bigl\} ,  m\rightarrow \infty$$ we get the desired bound.

\end{proof}

Let $z= Q^{\delta}$. In the notations above we in fact have shown that $$\tau(n,z)\lesssim {m\choose \delta' m}$$ where  $$m=[\frac{\log Q}{{\log{\log Q}}}], \delta'= \min (\delta, \frac{1}{2}).$$

In fact the condition $\mu(n)\neq 0$ in the previous lemma can be removed. 
\begin{lemma} \label{l-tau_nemu}
Let $n\leq Q $, $z\leq Q^{\delta}$. Then we have the following estimate $$\tau(n,z) \lesssim \exp\Bigl\{(\delta' \log (\frac{1}{\delta'}) + (1-\delta') \log (\frac{1}{1-\delta'})) \frac{\log Q}{{\log{\log Q}}}\Bigl\} ,  Q\rightarrow \infty, $$ where $\delta'= \min (\delta, \frac{1}{2})$.
\end{lemma}

\begin{proof}
We may assume that $\delta < \frac{1}{2}$,  as in the opposite situation this Lemma easily follows from the general estimate for $\tau(n)$. 
Let $\varepsilon>0$ be small fixed real number and $K$ be fixed large integer.

The proof consists of several steps and we begin with the first one.

\textbf{Step 1.}

We show that there exists a presentation of $n$ in the following form 
$$n=n_{1}\ldots n_{s}m,$$ where $\mu(n_{i})\neq 0,n_{i}>Q^{\varepsilon}, rad (m)\leq Q^{\varepsilon}$ (We allow the situation with $s=0$, where there are no $n_{i}$ in this presentation.) The argument of the proof is a sort of an algorithm. If $rad(n)\leq Q^{\varepsilon}$ then we are done with $n=m$. If $rad(n)>Q^{\varepsilon}$ then $n=rad(n) \frac{n}{rad(n)}$ and we proceed this procedure with  $\frac{n}{rad(n)}$ instead of $n$. It is easy to see that the algorithm will be finished and we get the desired representation.

\textbf{Step 2.}
We can easily get an upper estimate for the number of divisors of $m$. Indeed $m \leq Q,rad (n) \leq Q^{\varepsilon}$. We use Lemma \ref{l-3} and see that $$\tau(n) \lesssim  \exp\Bigl\{(C(\varepsilon)) \frac{\log Q}{{\log{\log Q}}}\Bigl\} ,  Q\rightarrow \infty,$$ where $C(\varepsilon)\longrightarrow 0$ if $\varepsilon\longrightarrow 0.$

\textbf{Step 3.} In this step we introduce some definitions. Let the quantities $\delta_{i}$ be defined from the identities $$n_{i}=Q^{\delta_{i}}.$$ Now we define $Ks$ intervals $\Omega_{i,j}$ $1\leq i \leq s, 1\leq j \leq K$ by setting $$\Omega_{i,j}=[Q^{\delta_{i}\frac{j-1}{K}}, Q^{\delta_{i}\frac{j}{K}}].$$

\textbf{Step 4.}
Any divisor of $n_{1}\ldots n_{s}$ can be presented as $d_{1}\ldots d_{s},$ $d_{i}|n_{i}$. Suppose that for any $1 \leq i \leq s$ we fix the interval $\Omega_{i,j_{i}}$. Now we will obtain upper estimate for the number of vectors $(d_{1}, \ldots , d_{s})$, $d_{i}|n_{i}$ and $d_{i} \in \Omega_{i,j_{i}}$.
Each $d_{i}$ is a divisor of $n_{i}$,  $\mu(n_{i}) \neq 0$,  $d_{i}\leq Q^{\delta_{i}\frac{j}{K}}$. So the number of such $d_{i}$ by Lemma \ref{l-tau_mu} does not exceed  $\lesssim {m_{i}\choose \delta_{i,j_{i}}' m_{i}}$ where $m_{i}=[\frac{ \delta_{i} \log Q}{{\log{\log Q}}}]$ and $\delta_{i,j_{i}}'= \min(\frac{1}{2}, \frac{j}{K})$. Therefore the number of such vectors $(d_{1}\ldots d_{s})$ is bounded by the product 

$$\lesssim \prod_{1 \leq i\leq s} {m_{i}\choose \delta_{i,j_{i}}' m_{i}} \leq {\sum_{i}m_{i}\choose \sum_{i}\delta_{i,j_{i}}' m_{i}}.$$

It is easy to see that $\sum_{i}m_{i}\leq \frac{\log Q}{{\log{\log Q}}}$. Next we are going to estimate $\sum_{i}\delta_{i,j_{i}}' m_{i}$.

We see that $$\sum_{i}\delta_{i,j_{i}}' m_{i}\leq \frac{1}{{\log{\log Q}}}\sum_{i}\delta_{i} \delta_{i,j_{i}}' \log Q    \quad (*)$$

Now we estimate each term in the last sum.

We have $$\log d_{i} \in [\delta_{i}\frac{j-1}{K} \log Q, \delta_{i}\frac{j}{K} \log Q]$$ and 
$$\delta_{i, j_{i}}' \leq \frac{j_{i}}{K}.$$ So we can write $$\delta_{i,j_{i}}' \delta_i \log Q \leq \delta_{i} \frac{j_{i}}{K} \log Q \leq \log d_{i} + \frac{\delta_{i}}{K} \log Q.$$  

Inserting this inequality to the expression (*) we obtain 

$$\sum_{i}\delta_{i,j_{i}}' m_{i}\leq \frac{\delta \log Q}{{\log{\log Q}}} + \frac{\delta \log Q}{K {\log{\log Q}}}$$

Finally we have that the number of such vectors $(d_{1},\ldots , d_{s})$ such that $d_{i}|n_{i}, d_{i} \in \Omega_{i,j}$ (the sets $\Omega_{i,j}$ are fixed) is bounded by $$\lesssim {M\choose \alpha M},$$ where

$$M=\frac{\log Q}{{\log{\log Q}}},\alpha = \min( \frac{1}{2}, \delta(1+\frac{1}{K})).$$ 

\textbf{Step 5.} Now we obtain an upper bound for the number of different choices of the intervals $\Omega_{i,j}$. This number does not exceed $K^{s}$ and is some bounded constant, (which does not depends on $Q$). Our Lemma now follows if one uses statements of Steps 2,4,5 and takes sufficiently large constant $K$ and sufficiently small $\varepsilon$.    



\end{proof}

\begin{lemma} \label{l-simple}
For any integer $n>1$ we have $\frac{\log \tau(n^{2})}{\log \tau(n)}\leq \frac{\log 3}{\log 2}.$

\end{lemma}
 
\begin{proof}
Let $n=p_{1}^{\beta_{1}}\ldots p_{l}^{\beta_{l}}$, then 
$$\frac{\log \tau(n^{2})}{\log \tau(n)}=\frac{\log(1+2\beta_{1}) + \ldots + \log(1+2\beta_{l})}{\log(1+\beta_{1}) + \ldots + \log(1+\beta_{l})}.$$ The last expression is always less than $\frac{\log 3}{\log 2}$. With that we finish the proof of this lemma.
 
\end{proof} 
Next, we introduce some notations.  Let $n \in \mathbb{N}$ and let $l(n)$ denotes the maximal positive integer $m$ such that $m^{2}|n$.



We are going to prove the following lemma.

\begin{lemma} \label{l-4}

Let  $n$ be positive integer, $n\leq Q^{2}$ and the quantity  $c$ is defined  from the equation
$$\tau(n)=\exp \Bigl\{ (2 \log 2  - c ) \frac{  \log Q}{\log\log Q} \Bigr\}. $$ Then there we have

$$l(n)  \leq Q^{\delta(c) + o(1)}, Q\rightarrow \infty, $$ where
$$\delta(c) \leq \frac{ c}{2  \log 2 - \log 3} .$$


\end{lemma}

\begin{proof}
 
Let the quantity $\delta$ is defined from the equality $l(n)=Q^{\delta}$. We also may assume that $$\log \log l(n)= (1+o(1)) \log \log Q,$$ as in the opposite situation the Lemma \ref{l-4} is true.

Using Lemma \ref{l-simple} and upper  estimate for $\tau(l(n))$ we conclude that 
$$\tau(l^{2}(n)) \leq \exp \Bigl\{ ( \log 3 + o(1)   ) \frac{ \delta \log Q}{\log\log Q} \Bigr\}, Q\rightarrow \infty.$$

We see that $l(n)^{2}|n$ and  we can write 

$$\exp \Bigl\{ (2 \log 2  - c ) \frac{  \log Q}{\log\log Q} \Bigr\}= \tau(n) \leq \tau(l(n)^{2}) \tau (\frac{n}{l(n)^{2}}).$$

It is easy to see that the last expression does not exceed 
$$
\exp \Bigl\{ \bigl( {\delta \log 3} + (2- 2 \delta) \log 2 +o(1) \bigr)  \frac{  \log Q}{\log\log Q} \Bigr\}.
$$
  
Comparing this quantity with the left-side expression in the last inequality and doing some easy computations we obtain the desired estimate for $\delta$. With that we finish the proof of Lemma \ref{l-4}. 
%




\end{proof}
Now we are ready to prove Theorem \ref{t-1} and we are going to the next section.


\section{The proof of Theorem \ref{t-1}}
\begin{proof}
Let the quantity $c$ be defined from the equality 
$$E(A,B)= |A||B| \exp \Bigl\{ ( 2 \log 2  - c ) \frac{  \log Q}{\log\log Q} \Bigr\}. $$
From the inequality \ref{shn} we see that 
$$|A/B| \geq  |A||B| \exp \Bigl\{ (- 2 \log 2  + c ) \frac{  \log Q}{\log\log Q} \Bigr\}.$$

Our next step is to find another lower bound for  $|A/B|$, -- this inequality will work well in the case of small $c$.


 

   
   Let us denote the quantity $L$ from the identity $E(A,B)=|A||B| L$, and let $$r_{A,B}(z)=|\{(a_{1},b_{1}) \in A \times B : a_{1}b_{1}= z\}|.$$
   
   Define the set $$M_{1}= \{z \in AB: r_{A,B}(z)\leq L/2\}$$ and $$M_{2}=AB \setminus M_{1}.$$

We see that $$\sum_{z \in M_{1}}r_{A,B}^{2}(z)\leq |A||B| L/2;$$
and so $$\sum_{z \in M_{2}}r_{A,B}^{2}(z)\geq |A||B| L/2.$$

For integer $i\geq 0$ let  $$M_{2,i}=\{z \in AB: r_{A,B}(z) \in (2^{i-1}L, 2^{i}L] \}.$$
We see that $$M_{2 }= \bigcup_{0 \leq i \ll \frac{\log Q}{\log \log Q}} M_{2,i}.$$

Hence by pigeonhole principle there exists $0 \leq i \ll \frac{\log Q}{\log \log Q}$ such that 

$$\sum_{z \in M_{2,i}}r_{A,B}^{2}(z)\geq |A||B| \exp \Bigl\{ (2 \log 2 - c + o(1))
   \frac{\log Q}{\log\log Q} \Bigr\}.$$  Let us fix such $i$ and let the quantity $c'$ be defined from the identity $$2^{i}L= \exp \Bigl\{ (2 \log 2 - c' )
   \frac{\log Q}{\log\log Q} \Bigr\}.$$

It is easy to see that $c' \in [o(1), c]$.

Next we  will show that $$\sum_{z \in M_{2,i}}r_{A,B}(z)  \geq |A||B| \exp
  \Bigl\{ (c' -c + o(1))
   \frac{\log Q}{\log\log Q} \Bigr\}, 
   Q\rightarrow \infty.$$   
   
Indeed, $$|A||B| \exp
  \Bigl\{ ( 2 \log 2 -c + o(1))
   \frac{\log Q}{\log\log Q} \Bigr\}\leq \sum_{z \in M_{2,i}}r_{A,B}^{2}(z) \leq \max_{z \in M_{2,i}} r_{A,B}(z) \sum_{z \in M_{2,i}}r_{A,B}(z).$$
 
The quantity  $\max_{z\in M_{2,i}} r_{A,B}(z)$ is  less than $\exp \Bigl\{ (2 \log 2 - c' )
   \frac{\log Q}{\log\log Q} \Bigr\}$. So  inserting this bound to the previous inequality we get the desired estimate.

Next we consider the set $G$:
$$G= \{(a_{1},b_{1}) \in A \times B: a_{1}b_{1} \in M_{2,i}\}.$$

From the previous estimate $|G|>  |A||B| \exp
  \Bigl\{ ( c'-c + o(1))
   \frac{\log Q}{\log\log Q} \Bigr\}, 
   Q\rightarrow \infty.$

Next we consider the following set $$W=\Bigl\{\frac{a_{1}}{b_{1}}: (a_{1}, b_{1}) \in G \Bigr\},$$ and will show that  $|W|$ is large.
   
For every element $z \in M_{2,i}$ we use Lemma \ref{l-4} and see that $$l(z)= Q^{\delta(c') + o(1)}, Q\longrightarrow \infty, $$ where $$\delta(c') \leq \frac{c'}{2  \log 2 - \log 3}.$$ 

This means that  for every pair $(a_{1},b_{1}) \in G$ $\gcd(a_{1},b_{1}) \leq Q^{\delta(c') + o(1)}.$

Define $$r_{A/B, G}(z)= \Bigl\{ (a_{1},b_{1}) \in G: \frac{a_{1}}{b_{1}} =z \Bigr\}.$$

We can write $$|A||B|\exp \Bigl\{ (c' - c +  o(1))
   \frac{\log Q}{\log\log Q} \Bigr\}= |G| = \sum_{z}r_{A/B, G}(z) \leq $$ 
   
   $$\leq  |W|^{1/2} \Bigl\{ \sum_{z} r_{A/B, G}^{2}(z) \Bigr\}^{1/2}.$$
   
   Our aim is to obtain good upper estimate for $$\sigma= \sum_{z} r_{A/B, G}^{2}(z).$$

   The $\sigma$ does not exceed the number of solution to the equation $$\frac{a_{1}}{b_{1}}=\frac{a_{2}}{b_{2}}, a_{i} \in A; b_{i} \in B \gcd(a_{1},b_{1}), \gcd(a_{2},b_{2}) \leq Q^{\delta(c') + o(1)}. $$

We may write $$a_{1} = tu, b_{1}=tv, a_{2}=su, b_{2}=sv ;$$
 where $$\gcd(u,v)=1 \quad  \textit{and }  \quad t,s \leq  Q^{\delta(c') +o(1)}.$$ 
 
Let us fix $a_{1}$ and $b_{2}$. If for these $a_{1}$ and $b_{2}$ we choose  $t$ and $s$ we then identify  $a_{2}$ and $a_{3}$. For any fixed $a_{1},b_{2}$ the parameters $t,s$ are the divisors of $a_{1},b_{2}$ respectively. These $t,s$ do not exceed $Q^{\delta(c') + o(1)}.$  Using Lemma \ref{l-tau_nemu}  we see that the number of different pairs  $t,s$ do not exceed 

$$ \exp\Bigl\{(2 \delta(c') \log (\frac{1}{\delta(c')}) + 2 (1-\delta(c')) \log (\frac{1}{1-\delta(c')})+o(1)) \frac{\log Q}{{\log{\log Q}}}\Bigl\} ,  Q\rightarrow \infty.$$
We will just write $\delta_{c'}$ instead of $\delta(c')$.

   
 And so we conclude that $$\sigma < |A||B|\exp\Bigl\{(2 \delta_{c'} \log (\frac{1}{\delta_{c'}}) + 2 (1-\delta_{c'}) \log (\frac{1}{1-\delta_{c'}})+o(1)) \frac{\log Q}{{\log{\log Q}}}\Bigl\} ,  Q\rightarrow \infty.$$
   
So we can obtain the lower bound for $|W|$:    
   
$$|W|\geq |A||B| \exp
  \Bigl\{ (2c'-2c -  2 \delta_{c'} \log (\frac{1}{\delta_{c'}}) - 2 (1-\delta_{c'}) \log (\frac{1}{1-\delta_{c'}})+o(1)) 
    \frac{\log Q}{\log\log Q} \Bigr\}, 
   Q\rightarrow \infty.$$ 

Recall that $c' \in [o(1), c]$.

We may assume $c\leq 0.11$. It is easy to see that the expression $$2c' - 2 \delta_{c'} \log (\frac{1}{\delta_{c'}}) - 2 (1-\delta_{c'}) \log (\frac{1}{1-\delta_{c'}})$$ takes the smallest value if  $c'=c$.

So, we can rewrite the last estimate     
$$|W|\geq |A||B| \exp
  \Bigl\{ ( - 2 \delta_{c} \log (\frac{1}{\delta_{c}}) - 2 (1-\delta_{c}) \log (\frac{1}{1-\delta_{c}}) + o(1))     \frac{\log Q}{\log\log Q} \Bigr\}, 
   Q\rightarrow \infty,$$
   
   where $\delta_{c}= \delta (c)$. 
 
As it was noted before there is trivial estimate
 $$|A/B| \geq  |A||B| \exp \Bigl\{ (- 2 \log 2  + c ) \frac{  \log Q}{\log\log Q} \Bigr\}.$$
 
We have these two estimates, one work well with small $c$, another work well with large $c$.  

It is easy to see that the explicit absolute constant $\gamma>0$ can be taken as the solution of the following equation

$$-2 \log 2 +c =  - 2 \delta_{c} \log (\frac{1}{\delta_{c}}) - 2 (1-\delta_{c}) \log (\frac{1}{1-\delta_{c}}),$$ 
where $\delta_{c}=\frac{c}{2\log2-\log3}$.

Computer calculations show that the solution is equal to $0.098....$, so one can this value for the $\gamma$.
  With that we finish the proof of Theorem \ref{t-1}.
\end{proof}

\section{Final remarks}

One can easily deduce the following corollary, which follows from the proof of Theorem~\ref{t-1}.

\begin{corollary}
Let $A,B \subseteq [1,Q]$ and  $E(A,B)= |A||B|\exp\Bigl\{( 2 \log2 + o(1)) \frac{\log Q}{{\log{\log Q}}}\Bigl\}.$ Then we have $$|A/B| = |A||B|\exp\Bigl\{(  o(1)) \frac{\log Q}{{\log{\log Q}}}\Bigl\} ,  Q\rightarrow \infty. $$
In particular if $|AB| = |A||B|\exp\Bigl\{(-2 \log 2 +  o(1)) \frac{\log Q}{{\log{\log Q}}}\Bigl\} ,  Q\rightarrow \infty, $ then  $$|A/B| = |A||B|\exp\Bigl\{(  o(1)) \frac{\log Q}{{\log{\log Q}}}\Bigl\} ,  Q\rightarrow \infty. $$
\end{corollary}

Indeed, the  the condition $|AB| = |A||B|\exp\Bigl\{(-2 \log 2 +  o(1)) \frac{\log Q}{{\log{\log Q}}}\Bigl\} $ imply $E(A,B)=|A||B|\exp\Bigl\{( 2 \log2 + o(1)) \frac{\log Q}{{\log{\log Q}}}\Bigl\}$. 

It seems that using more precise arguments for finding pairs $(a,b) \in A \times B $ for the set $G$ with smaller $\gcd(a,b)$ may lead to a better coefficient instead of $0.098....$

\section*{Acknowledgements} 

The work   is supported by the Russian Science Foundation under grant 14-50-00005.  
I wish  to thank  Sergei Konyagin for  valuable comments, advices and attention to this work.

Steklov Institute of Mathematics, Russian Academy of Science,  yuriisht@yandex.ru


\begin{thebibliography}{99}

\bibitem{sht}   Shteinikov Yu., On the product sets of rational numbers, Proceedings of the Steklov Institute of Mathematics, Vol. 296, 2017, preprint.

\bibitem{crr}   J. Cilleruelo, D.S.Ramana and O.Ramare,    Quotients
and    product    sets    of    thin    subsets
of
the
positive
integers, Proceedings of the Steklov Institute of Mathematics, Vol. 296, 2017, preprint.

\bibitem{prach}  Prachar K.  Primzahlverteilung // Springer--Verlag
Berlin--G\H ottingen--Heidelberg, 1957.

\bibitem{II}  Shteinikov Yu. N.   Addendum to the paper "Quotients and product
sets of thin subsets of the positive integers" by J.
Cilleruelo, D.S. Ramana and O. Ramare.  //  Proceedings of the Steklov Institute of Mathematics, 296, 2017, ....

\bibitem{ten}  Adolf Hildebrand, Gerald Tenenbaum.  Integers without  large prime factors // Journal de Theorie des Nombres de Bordeaux 5, (1993), 411-484.


\bibitem{rasp}  Bourgain J.,  Konyagin S.V.,   Shparlinski I.E.   Product sets of rationals, multiplicative translates of subgroups in residue rings and fixed points of the discrete logarithm // Int. Math Research Notices.  2008. rnn 090, P. 1--29.

\bibitem{C}   Cilleruelo J.  A note on product sets of rationals //  International  Journal of Number Theory, Vol. 12, No. 05, pp. 1415-1420 (2016)

\bibitem{CG} Cilleruelo J., Garaev M.  Congruences involving product of intervals and sets with small multiplicative doubling modulo a prime and applications  //  Math. Proc. Cambridge Phil. Soc.,  Vol. 160, Issue 03, pp 477-494, May 2016.

\bibitem{TV}  Tao T., Vu V.  Additive combinatorics //  Cambridge University Press 2006,~P.~1-530.

\bibitem{lg} Shnirel'man L.G. Uber additive Eigenschaften von Zahlen // Mathematische Annalen, V. 107 (1933), P. 649-690.



\end{thebibliography}
\end{document}